\documentclass{amsart}
\usepackage{amscd,amsthm,amsmath,amssymb}
\usepackage{textcomp}
\usepackage{hyperref}
\usepackage{color}
\usepackage{tikz,pgfplots,pgf} 
\tikzset{node distance=2cm, auto}
\usetikzlibrary{calc} 
\usetikzlibrary{trees} 
\newtheorem{theorem}{Theorem}
\newtheorem{corollary}[theorem]{Corollary}
\newtheorem{remark}[theorem]{Remark}
\newtheorem{example}[theorem]{Example}

\def\N{\mathbb{N}}
\def\Z{\mathbb{Z}}
\def\R{\mathbb{R}}

\def\K{\mathbb{K}}
\def\T{\mathbb{T}}
\def\diam{\mathrm{diam}}

\def\FX{\mathcal{F}(X)}
\def\Lipo{\mathrm{Lip}_0}
\def\lipo{\mathrm{lip}_0}

\def\Lip{\mathrm{Lip}}
\def\lip{\mathrm{lip}}

\begin{document}
\title[Biduality and density in Lipschitz function spaces]{Biduality and density in Lipschitz function spaces}

\author{A. Jim\'{e}nez-Vargas}
\address{Departamento de Matem\'{a}ticas, Universidad de Almer\'{i}a, 04120 Almer\'{i}a, Spain}
\email{ajimenez@ual.es}
\date{\today}
\thanks{This research was partially supported by MICINN under project MTM 2010-17687.}

\author{J.M. Sepulcre}
\address{Departamento de An\'{a}lisis Matem\'{a}tico, Universidad de Alicante, 03080 Alicante, Spain}
\email{JM.Sepulcre@ua.es}

\author{Mois\'{e}s  Villegas-Vallecillos}
\address{Campus Universitario de Puerto Real, Facultad de Ciencias, 11510 Puerto Real, C\'{a}diz, Spain}
\email{moises.villegas@uca.es}

\subjclass[2010]{46E10, 46E15, 46J10}
\keywords{Lipschitz function, little Lipschitz function, H\"{o}lder function, Lipschitz-free Banach space}

\begin{abstract}
For pointed compact metric spaces $(X,d)$, we address the biduality problem as to when the space of Lipschitz functions $\Lipo(X,d)$ is isometrically isomorphic to the bidual of the space of little Lipschitz function $\lipo(X,d)$, and show that this is the case whenever the closed unit ball of $\lipo(X,d)$ is dense in the closed unit ball of $\Lipo(X,d)$ with respect to the topology of pointwise convergence. Then we apply our density criterion to prove in an alternate way the real version of a classical result which asserts that $\Lipo(X,d^\alpha)$ is isometrically isomorphic to $\lipo(X,d^\alpha)^{**}$ for any $\alpha\in(0,1)$. 
\end{abstract}
\maketitle

\section*{Introduction}
Let $(X,d)$ be a pointed compact metric space with a base point denoted by $0$ and let $\mathbb{K}$ be the field of real or complex numbers. The Lipschitz space $\Lipo(X,d)$ is the Banach space of all Lipschitz functions $f\colon X\to\K$ for which $f(0)=0$, endowed with the Lipschitz norm
$$
\Lip_d(f)=\sup\left\{\frac{\left|f(x)-f(y)\right|}{d(x,y)}\colon x,y\in X, \; x\neq y\right\}.
$$
A Lipschitz function $f\colon X\to\K$ satisfying the local flatness condition: 
$$
\lim_{t\to 0}\sup_{0<d(x,y)<t}\frac{\left\|f(x)-f(y)\right\|}{d(x,y)}=0,
$$
is called a little Lipschitz function, and the little Lipschitz space $\lipo(X,d)$ is the closed subspace of $\Lipo(X,d)$ formed by all little Lipschitz functions. Furthermore, $\Lipo^\R(X,d)$ and $\lipo^\R(X,d)$ are the real subspaces of all real-valued functions in $\Lipo(X,d)$ and $\lipo(X,d)$, respectively. These spaces have been largely investigated along the time. See the Weaver's book \cite{weaver} for references and a complete study.

The biduality problem as to when $\Lipo(X,d)$ is isometrically isomorphic to $\lipo(X,d) ^{**}$ has an interesting history (see \cite[p. 99, Notes 3.3]{weaver} and also \cite[6. Duality]{k04}). In this note, we address this question in a similar way as Bierstedt and Summers \cite{bs} do for studying the biduals of weighted Banach spaces of analytic functions, and we prove that $\Lipo(X,d)$ is isometrically isomorphic to $\lipo(X,d) ^{**}$ if and only if the closed unit ball of $\lipo(X,d)$ is dense in the closed unit ball of $\Lipo(X,d)$ with respect to the topology of pointwise convergence $\tau_p$. This density condition is equivalent to requiring that for each $f\in\Lipo(X,d)$ with $\Lip_d(f)\leq 1$, there exists a sequence $\{f_n\}$ in $\lipo(X,d)$ with $\Lip_d(f_n)\leq 1$ for all $n\in\N$ such that $\{f_n(x)\}$ converges to $f(x)$ as $n\to\infty$ for every $x\in X$. Then we apply our density criterion to prove in an alternate way the real version of a classical Johnson's result \cite{j70} (see also \cite{bcd,d,weaver}) which asserts that $\Lipo(X,d^\alpha)$ is isometrically isomorphic to $\lipo(X,d^\alpha) ^{**}$ for any $\alpha\in(0,1)$. 

\section*{The results}

Johnson \cite{j70} proved that the closed linear subspace of $\Lipo(X,d)^*$ spanned by the evaluation functionals $\delta_x\colon\Lipo(X,d)\to\K$, given by
$\delta_x(f)=f(x)$ with $x\in X$, is a predual of $\Lipo(X,d)$. The terminology Lipschitz-free Banach space of $X$ and the notation $\FX$ for this predual of $\Lipo(X,d)$ are due to Godefroy and Kalton \cite{gk}. Namely, the evaluation map $Q_X\colon\Lipo(X,d)\to\FX^*$ defined by
$$
Q_X(f)(\gamma)=\gamma(f) \qquad \left(f\in\Lipo(X,d), \; \gamma\in\FX \right)
$$
is the natural isometric isomorphism. As usual, $B_E$ will denote the closed unit ball of a Banach space $E$.

\begin{theorem}\label{messi1}
Let $(X,d)$ be a pointed compact metric space. 
\begin{enumerate}
\item The restriction map $R_X\colon\FX\to\lipo(X,d)^{*}$ defined by
$$
R_X(\gamma)(f)=\gamma(f) \qquad \left(f\in\lipo(X,d), \; \gamma\in\FX \right),
$$
is a nonexpansive linear surjective map.
\item $R_X$ is an isometric isomorphism from $\FX$ onto $\lipo(X,d)^{*}$ if and only if $B_{\lipo(X,d)}$ is dense in $B_{\Lipo(X,d)}$ with respect to the topology of pointwise convergence.
\end{enumerate}
\end{theorem}

\begin{proof}
(i) Since $\FX\subset\Lipo(X,d)^{*}$, it is clear that $R_X$ is a linear map from $\FX$ into $\lipo(X,d)^{*}$ satisfying $\left\|R_X(\gamma)\right\|\leq\left\|\gamma\right\|$ for all $\gamma\in\FX$. We next prove that $R_X$ is surjective. To this end, let us recall that De Leeuw's map $\Phi\colon\Lipo(X,d)\to C_b(\widetilde{X})$ given by 
$$
\Phi(f)(x,y)=\frac{f(x)-f(y)}{d(x,y)} \qquad (f\in\Lipo(X,d), \; (x,y)\in\widetilde{X}),
$$
where $\widetilde{X}=\left\{(x,y)\in X^2\colon x\neq y\right\}$, is a linear isometry of $\Lipo(X,d)$ into $C_b(\widetilde{X})$, the Banach space of bounded continuous scalar-valued functions on $\widetilde{X}$ with the supremum norm, and the image of $\lipo(X,d)$ is contained in $C_0(\widetilde{X})$, the closed subspace of functions which vanish at infinity. See, for example, \cite[Theorem 2.1.3 and Proposition 3.1.2]{weaver}.

Take $\gamma\in\lipo(X,d)^{*}$. The functional $T\colon\Phi(\lipo(X,d))\to\K$, defined by $T(\Phi(f))=\gamma(f)$ for all $f\in\lipo(X,d)$, is linear, continuous and $\left\|T\right\|=\left\|\gamma\right\|$. By the Hahn--Banach theorem, there exists a continuous linear functional $\widetilde{T}\colon\mathcal{C}_0(\widetilde{X})\to\K$ such that $\widetilde{T}(\Phi(f))=T(\Phi(f))$ for all $f\in\lipo(X,d)$ and $||\widetilde{T}||=\left\|T\right\|$. Now, by the Riesz representation theorem, there exists a finite and regular Borel measure $\mu$ on $\widetilde{X}$ with total variation $\left\|\mu\right\|=||\widetilde{T}||$ such that
$$
\widetilde{T}(g)=\int_{\widetilde{X}}g\, d\mu \qquad (g\in\mathcal{C}_0(\widetilde{X})),
$$
and thus
$$
\gamma(f)=\int_{\widetilde{X}}\Phi(f)\, d\mu \qquad (f\in\lipo(X,d)).
$$
If we now define
$$
\widetilde{\gamma}(f)=\int_{\widetilde{X}}\Phi(f)\, d\mu \qquad (f\in\Lipo(X,d)),
$$
it is clear that $\widetilde{\gamma}\in\Lipo(X,d)^*$ and $\widetilde{\gamma}(f)=\gamma(f)$ for all $f\in\lipo(X,d)$. 
It remains to show that $\widetilde{\gamma}$ is $\tau_p$-continuous on $B_{\Lipo(X,d)}$. Thus, let $\{f_i\}$ be a net in $B_{\Lipo(X,d)}$ which converges pointwise on $X$ to zero. Then $\{\Phi(f_i)\}$ converges pointwise on $\widetilde{X}$ to zero 
and, since $\left|\Phi(f_i)(x,y)\right|\leq\left\|\Phi(f_i)\right\|_{\infty}=\Lip_d(f_i)\leq 1$ for all $i\in I$ and for all $(x,y)\in\widetilde{X}$, it follows that $\{\widetilde{\gamma}(f_i)\}$ converges to $0$ by the Lebesgue's bounded convergence theorem. This completes the proof of (i).

(ii) Assume that $B_{\lipo(X,d)}$ is $\tau_p$-dense in $B_{\Lipo(X,d)}$. Fix $\gamma\in\FX$ and let $f\in B_{\Lipo(X,d)}$. Then there exists a net $\{f_i\}$ in $B_{\lipo(X,d)}$ which converges to $f$ in the topology of pointwise convergence. Since $\gamma$ is $\tau_p$-continuous on $B_{\Lipo(X,d)}$ and it is satisfied that 
$$
\left|\gamma(f_i)\right|
=\left|R_X(\gamma)(f_i)\right|
\leq\left\|R_X(\gamma)\right\|\Lip_d(f_i)
\leq\left\|R_X(\gamma)\right\|
$$
for all $i\in I$, it follows that $\left|\gamma(f)\right|\leq\left\|R_X(\gamma)\right\|$ and so $\left\|\gamma\right\|\leq\left\|R_X(\gamma)\right\|$. Now, taking into account (i) we conclude that $R_X$ is an isometric isomorphism from $\FX$ onto $\lipo(X,d)^{*}$.

Conversely, if $B_{\lipo(X,d)}$ is not $\tau_p$-dense in $B_{\Lipo(X,d)}$,  
by the Hahn--Banach theorem there exist a function $g\in B_{\Lipo(X,d)}$ and a $\tau_p$-continuous linear functional $\gamma$ on $\Lipo(X,d)$ such that $\left|\gamma(f)\right|\leq 1$ for all $f\in B_{\lipo(X,d)}$ and $\left|\gamma(g)\right|>1$. Since $\gamma\in\FX$ and $\left\|R_X(\gamma)\right\|=\left\|\left.\gamma\right|_{\lipo(X,d)}\right\|\leq 1<\left|\gamma(g)\right|\leq\left\|\gamma\right\|$, then $R_X$ is not an isometry. 
\end{proof}

We are now ready to obtain the main result of this note.

\begin{theorem}\label{messi}
Let $(X,d)$ be a pointed compact metric space. Then the following are equivalent:
\begin{enumerate}
	\item $\Lipo(X,d)$ is isometrically isomorphic to $\lipo(X,d) ^{**}$.
	\item $B_{\lipo(X,d)}$ is dense in $B_{\Lipo(X,d)}$ with respect to the weak* topology.
	\item $B_{\lipo(X,d)}$ is dense in $B_{\Lipo(X,d)}$ with respect to the topology of pointwise convergence.
	\item For each $f\in B_{\Lipo(X,d)}$, there exists a sequence $\{f_n\}$ in $B_{\lipo(X,d)}$ such that $\{f_n(x)\}$ converges to $f(x)$ as $n\to\infty$ for every $x\in X$.
\end{enumerate}
\end{theorem}

\begin{proof}
If (i) holds, then (ii) follows by the Goldstine theorem; but (ii) is the same as (iii) since the weak* topology agrees with the topology of pointwise convergence on bounded subsets of $\Lipo(X,d)$ by \cite[Corollary 4.4]{j70}. If (iii) is true, then $R_X^*$ is an isometric isomorphism from $\lipo(X,d)^{**}$ onto $\FX^*$ by Theorem \ref{messi1}, hence the composition $Q_X^{-1}\circ R_X^*$ is an isometric isomorphism from $\lipo(X,d)^{**}$ onto $\Lipo(X,d)$ and so we obtain (i). 

In order to prove that (ii) is equivalent to (iv), notice that, by \cite[Corollary 4.4]{j70}, the family of sets 
$$
U(f_0;n,x_1,\ldots,x_n,\varepsilon):=\left\{f\in B_{\Lipo(X,d)}\colon \left|f(x_i)-f_0(x_i)\right|<\varepsilon, \; \forall i=1,\ldots,n\right\}
$$
with $f_0\in B_{\Lipo(X,d)}$, $n\in\N$, $x_1,\ldots,x_n\in X$ and $\varepsilon>0$, is a basis of the relative weak* topology on $B_{\Lipo(X,d)}$. 

Suppose now that (ii) holds and let $f_0\in B_{\Lipo(X,d)}$. Given $x\in X$ and $n\in\N$, the set $U(f_0;1,x,1/n)$ is a weak* neighborhood of $f_0$ relative to $B_{\Lipo(X,d)}$. Then, by (ii), for each $n\in\N$ there exists $f_n\in B_{\lipo(X,d)}$ such that $f_n\in U(f_0;1,x,1/n)$, that is, $\left|f_n(x)-f_0(x)\right|<1/n$. Hence $\{f_n(x)\}$ converges to $f_0(x)$ as $n\to\infty$ and we conclude that (ii) implies (iv). Conversely, assume that (iv) is valid and let $f_0\in B_{\Lipo(X,d)}$. Take $U(f_0;p,x_1,\ldots,x_p,\varepsilon)$ with $p\in\N$, $x_1,\ldots,x_p\in X$ and $\varepsilon>0$. By (iv), there is a sequence $\{f_n\}$ in $B_{\lipo(X,d)}$ such that $\{f_n(x)\}$ converges to $f_0(x)$ as $n\to\infty$ for every $x\in X$. In particular, for each $i\in\{1,\ldots,p\}$, there is a $m_i\in\N$ for which $\left|f_n(x_i)-f_0(x_i)\right|<\varepsilon$ whenever $n\geq m_i$. Now, if $m=\max\{m_1,\ldots,m_p\}$, then $f_m\in U(f_0;p,x_1,\ldots,x_p,\varepsilon)$ and (ii) follows. 
\end{proof}

It is known that $\Lipo(X,d)$ is isometrically isomorphic to $\lipo(X,d) ^{**}$ for a large class of metric spaces $(X,d)$ as, for example, the H\"{o}lder spaces $(X,d^\alpha)$, $0<\alpha<1$ \cite{bcd,d,j70}. 

\begin{remark}
The proof of Theorem \ref{messi} shows that if one of its statements holds, then the map $Q_X^{-1}\circ R_X^*$ is an isometric isomorphism from $\lipo(X,d)^{**}$ onto $\Lipo(X,d)$. For any $\phi\in\lipo(X,d)^{**}$ and $x\in X$, an easy verifications yields
\begin{align*}
(Q_X^{-1}\circ R_X^*)(\phi)(x)&=\delta_x((Q_X^{-1}\circ R_X^*)(\phi))\\
                              &=Q_X((Q_X^{-1}\circ R_X^*)(\phi))(\delta_x)\\
                              &=Q_X(Q_X^{-1}(R_X^*(\phi)))(\delta_x)\\
                              &=R_X^*(\phi)(\delta_x)\\
                              &=\phi(R_X(\delta_x))\\
                              &=\phi(\delta_x).
\end{align*}
This identification is the same as that obtained by De Leeuw \cite{d}, Johnson \cite{j70} and Bade, Curtis and Dales \cite{bcd} between the spaces $\Lipo(X,d^\alpha)$ and $\lipo(X,d^\alpha)^{**}$ $(0<\alpha<1)$.
\end{remark}

The pointwise approximation condition given by the assertion (iv) of Theorem \ref{messi} can be verified to recover two classical results about the biduality problem of $\Lipo(X,d^\alpha)$ $(0<\alpha<1)$. The former is due to Ciesielski \cite{ci} and the latter to De Leeuw \cite{d}. 

\begin{example}
Let $\alpha\in (0,1)$ and let $[0,1]$ be the unit interval with the usual metric $d$. Then $\Lipo([0,1],d^\alpha)$ is isometrically isomorphic to $\lipo([0,1],d^\alpha)^{**}$.
\end{example}

\begin{proof}
Fix $f\in B_{\Lipo([0,1],d^\alpha]}$ and, for each $n\in\N$, let $B_n(f,\cdot)$ denote the $n$th Bernstein polynomial for $f$ defined by
$$
B_n(f,x)=\sum_{k=0}^{n}f\left(\frac{k}{n}\right)\binom{n}{k}x^k(1-x)^{n-k}\qquad (x\in [0,1]).
$$
Then $B_n(f,\cdot)$ also belongs to $B_{\Lipo([0,1],d^\alpha)}$ (see \cite{bep} for an elementary proof) while the fact that 
\begin{align*}
\left|B_n(f,x)-B_n(f,y)\right|
&\leq\sum_{k=0}^{n}\left|f\left(\frac{k}{n}\right)\right|\binom{n}{k}\left|x^k(1-x)^{n-k}-y^k(1-y)^{n-k}\right|\\
&\leq\left|x-y\right|\sum_{k=0}^{n}\left|f\left(\frac{k}{n}\right)\right|\binom{n}{k}2n
\end{align*}
for all $x,y\in [0,1]$ shows that $B_n(f,\cdot)\in B_{\lipo([0,1],d^\alpha)}$. Since $\left\{B_n(f,\cdot)\right\}_{n\in\N}$ converge to $f$ uniformly on $[0,1]$, the example is proved by Theorem \ref{messi}.
\end{proof}

\begin{example}
Let $0<\alpha<1$ and let $\T$ be the quotient additive group $\R/2\pi\Z$ with the distance 
$$
d(t+2\pi\Z,s+2\pi\Z)=\min\left\{\left|t-s\right|,\left|t-s-2\pi\right|,\left|t-s+2\pi\right|\right\}\qquad \left(t,s\in [0,2\pi)\right).
$$
Then $\Lipo(\T,d^\alpha)$ is isometrically isomorphic to $\lipo(\T,d^\alpha)^{**}$.
\end{example}

\begin{proof}
We apply similar arguments to those of \cite[Lemma 2.8]{d} and use some results from harmonic analysis (see \cite{ka}). We identify each equivalence class $t+2\pi\Z$ with the point $t\in [0,2\pi)$. Let $f\in B_{\Lipo(\T,d^\alpha)}$. For each $n\in\N$, let $K_n$ be the Fej\'{e}r kernel defined by
$$
K_n(t)=\sum_{j=-n}^{n}\left(1-\frac{\left|j\right|}{n+1}\right)e^{ijt}=\frac{1}{n+1}\left(\frac{\sin\frac{n+1}{2}}{\sin\frac{t}{2}}\right)^2\qquad (t\in [0,2\pi)).
$$
Then the convolution 
$$
(K_n\ast f)(t)=\frac{1}{2\pi}\int_{0}^{2\pi}K_n(\tau)f(t-\tau)\ d\tau \qquad (t\in [0,2\pi))
$$
coincides with the Fej\'{e}r mean 
$$
\sigma_n(f,t)=\sum_{j=-n}^{n}\left(1-\frac{\left|j\right|}{n+1}\right)\widehat{f}(j)e^{ijt}\qquad (t\in [0,2\pi)),
$$
where $\widehat{f}(j)$ is the $j$th Fourier coefficient of $f$. Given $t,s\in [0,2\pi)$, we have
\begin{align*}
\left|\sigma_n(f,t)-\sigma_n(f,s)\right|
&\leq\sum_{j=-n}^{n}\left|1-\frac{\left|j\right|}{n+1}\right||\widehat{f}(j)|\left|e^{ijt}-e^{ijs}\right|\\
&\leq\sum_{j=-n}^{n}\left|1-\frac{\left|j\right|}{n+1}\right|\frac{\pi{^\alpha}|j|^{1-\alpha}}{2}(4\pi n)^n (e-1)d(t,s)
\end{align*}
and therefore $\sigma_n(f,\cdot)\in\lipo(\T,d^\alpha)$. 
Moreover,  
\begin{align*}
\left|\sigma_n(f,t)-\sigma_n(f,s)\right|
&=\left|(K_n\ast f)(t)-(K_n\ast f)(s)\right|\\
&\leq \frac{1}{2\pi}\int_{0}^{2\pi}\left|K_n(\tau)\right|\left|f(t-\tau)-f(s-\tau)\right|\ d\tau\\
&\leq \Lip_{d^\alpha}(f)d(t,s)^\alpha\frac{1}{2\pi}\int_{0}^{2\pi}K_n(\tau)\ d\tau\\
&=\Lip_{d^\alpha}(f)d(t,s)^\alpha ,
\end{align*}
and so $\Lip_{d^\alpha}(\sigma_n(f,\cdot))\leq \Lip_{d^\alpha}(f)\leq 1$. Now take $\beta_n(f,\cdot)=\sigma_n(f,\cdot)-\sigma_n(f,0)$ which is in $B_{\lipo(\T,d^\alpha)}$. By the Fej\'{e}r theorem, 
$\left\{\sigma_n(f,\cdot)\right\}_{n\in\N}$ converges pointwise on $\T$ to $f$, and so does $\left\{\beta_n(f,\cdot)\right\}_{n\in\N}$. Then the desired conclusion follows from Theorem \ref{messi}.
\end{proof}

Our density criterion serves to give another proof of the real version of an important result by Johnson \cite[Theorem 4.7]{j70} and Bade, Curtis and Dales \cite[Theorem 3.5]{bcd}. 

\begin{corollary}
Let $(X,d)$ be a pointed compact metric space and let $\alpha\in (0,1)$. Then $\Lipo^\R(X,d^\alpha)$ is isometrically isomorphic to $\lipo^\R(X,d^\alpha)^{**}$.
\end{corollary}

\begin{proof}
Let $f\in B_{\Lipo^\R(X,d^\alpha)}$. We claim that for each $n\in\N$ and each finite set $F\subset X$, there exists a function $h\in\lip^\R(X,d^\alpha)$ such that $\Lip_{d^\alpha}(h)\leq 1+1/n$ and $h(x)=f(x)$ for all $x\in F$. The notation $\lip^\R(X,d^\alpha)$ and later $\Lip^\R(X,{d^\gamma})$ might be self-explanatory.

Consider $F=\{x_1,\ldots,x_m\}$ for some $m\in\N$ and there is no loss of generality in assuming that $f(x_m)\leq f(x_{m-1})\leq\ldots\leq f(x_1)$. If $m=1$, we set $h(x)=f(x_1)$ for all $x\in X$. Now let $m\geq 2$ and we also may assume $f\geq 0$, for otherwise we can replace $f$ by $f+\left\|f\right\|_\infty$. Let
\begin{align*}
\gamma &=\min\left(\left\{\alpha+\frac{e\ln(1+\frac{1}{n})}{\diam(X)}d(x_j,x_k)\colon j,k\in\{1,\ldots,m\},\; j\neq k\right\}\cup \left\{1\right\}\right),\\
\rho &=\max\left\{\frac{\left|f(x_k)-f(x_j)\right|}{d(x_k,x_j)^\gamma}\colon j,k\in\{1,\ldots,m\},\; j\neq k\right\}.
\end{align*}
For each $j\in\{1,\ldots,m\}$, define $g_j\colon X\to\R$ by 
$$
g_j(x)=\max\left\{f(x_j)-\rho d(x_j,x)^\gamma , 0\right\}.
$$
Notice that $0<\alpha<\gamma\leq 1$ and therefore $g_j\in\Lip^\R(X,d^\gamma)\subset\lip^\R(X,d^\alpha)$ with 
$$
\Lip_{d^\alpha}(g_j)\leq\Lip_\gamma(g_j)\diam(X)^{\gamma-\alpha}\leq\rho\,\diam(X)^{\gamma-\alpha}.
$$
We now check that the function $h=\max\left\{g_1,\ldots ,g_m\right\}$ satisfies the required conditions. It is known that $h$ is in $\lip^\R(X,d^\alpha)$ and $\Lip_{d^\alpha}(h)\leq\max\left\{\Lip_{d^\alpha}(g_1),\ldots,\Lip_{d^\alpha}(g_m)\right\}$. Now, given $j\in \{1,\ldots,m\}$, for some $k,i\in\{1,\ldots,m\}$ with $k\neq i$, we have
\begin{align*}
\Lip_{d^\alpha}(g_j)&\leq\rho\,\diam(X)^{\gamma-\alpha}\\
             &=\frac{\left|f(x_k)-f(x_i)\right|}{d(x_k,x_i)^\gamma}\diam(X)^{\gamma-\alpha}\\
             &\leq \Lip_{d^\alpha}(f)\left(\frac{\diam(X)}{d(x_k,x_i)}\right)^{\gamma-\alpha}\\
             &\leq \left(\frac{\diam(X)}{d(x_k,x_i)}\right)^{\frac{e\ln(1+\frac{1}{n})}{\diam(X)}d(x_k,x_i)}\\
&\leq 1+\frac{1}{n}.
\end{align*}
The last inequality follows from the fact that the function $t\mapsto(t/\diam(X))^{te\ln(1+1/n)/\diam(X)}$ for all $t>0$
has a minimum value of $1/(1+1/n)$. Hence $\Lip_{d^\alpha}(h)\leq 1+1/n$ as required. Now let $j,k\in \{1,\ldots,m\}$. If $j\leq k$, it is immediate that $g_k(x_j)\leq f(x_k)\leq f(x_j)=g_j(x_j)$, whereas that if $k<j$, we have $\left|f(x_k)-f(x_j)\right|/d(x_k,x_j)^\gamma\leq \rho$, hence $f(x_k)-\rho d(x_k,x_j)^\gamma\leq f(x_j)$ and thus $g_k(x_j)\leq g_j(x_j)$. Therefore $h(x_j)=g_j(x_j)=f(x_j)$ for all $j\in \{1,\ldots,m\}$. The claim follows.
 
Now fix $n\in\N$ and, for each $x\in X$, let 
$B(x,1/n)=\left\{y\in X\colon d(y,x)^\alpha<1/n\right\}$. By the compactness of $X$, there is a finite subset $F_n$ of $X$ such that $X=\bigcup_{x\in F_n}B(x,1/n)$. We can suppose that the base point $0\in X$ is in $F_n$, for otherwise take the finite set $F_n\cup\{0\}$. By the claim, there exists a function $h_n\in\lip^\R(X,d^\alpha)$ such that $\Lip_{d^\alpha}(h_n)\leq 1+1/n$ and $h_n(x)=f(x)$ for all $x\in F_n$. Hence $h_n\in\lipo^\R(X,d^\alpha)$. To prove that the sequence $\{h_n\}$ converges pointwise on $X$ to $f$, let $x\in X$. For each $n\in\N$, choose $y_n\in F_n$ such that $d(x,y_n)^\alpha<1/n$. Note that $h_n(y_n)=f(y_n)$ and thus
\begin{align*}
\left|f(x)-h_n(x)\right|&\leq \left|f(x)-f(y_n)\right|+\left|f(y_n)-h_n(x)\right|\\
                          &\leq \left|f(x)-f(y_n)\right|+\left|h_n(y_n)-h_n(x)\right|\\
                          &\leq \left(\Lip_{d^\alpha}(f)+\Lip_{d^\alpha}(h_n)\right)d(x,y_n)^\alpha\\
                          &\leq \left(2+\frac{1}{n}\right)\frac{1}{n}.
\end{align*}
Hence the sequence $\{h_n(x)\}$ converges to $f(x)$ as $n\to\infty$. Finally, let $r_n=\max\{1,\Lip_{d^\alpha}(h_n)\}$ and $f_n=h_n/r_n$ for each $n\in\N$. It is clear that $\{f_n\}$ is a sequence in $B_{\lipo^\R(X,d^\alpha)}$ that converges pointwise to $f$ on $X$. Then the corollary follows from Theorem \ref{messi}.
\end{proof}

\bibliographystyle{amsplain}

\end{document}